\documentclass[a4paper, 12pt]{article}

\usepackage{amssymb, amscd, amsthm}
\usepackage[mathscr]{eucal}

\theoremstyle{plain}
\newtheorem{theorem}{Theorem}[section]
\newtheorem{lemma}[theorem]{Lemma}
\newtheorem{proposition}[theorem]{Proposition}
\newtheorem{corollary}[theorem]{Corollary}

\theoremstyle{definition}

\newtheorem{remark}[theorem]{Remark}
\newtheorem{remarks}[theorem]{Remarks}
\newtheorem{example}[theorem]{Example}

\newcommand\bA{{\mathbb A}}

\newcommand\bF{{\mathbb F}}
\newcommand\bG{{\mathbb G}}

\newcommand\bR{{\mathbb R}}

\newcommand\bZ{{\mathbb Z}}

\newcommand\cC{{\mathcal C}}

\newcommand\cM{{\mathcal M}}

\newcommand\cO{{\mathcal O}}

\newcommand\fm{\mathfrak{m}}

\newcommand\aff{{\rm aff}}

\newcommand\ant{{\rm ant}}

\newcommand\diag{{\rm diag}}

\newcommand\id{{\rm id}}

\newcommand\sq{{\rm sq}}

\newcommand\End{{\rm End}}

\newcommand\GL{{\rm GL}}

\newcommand\Hom{{\rm Hom}}
\newcommand\Ima{{\rm Im}}

\newcommand\Sym{{\rm Sym}}

\title{On Algebraic Semigroups and Monoids, II}

\author{Michel Brion}

\date{}

\begin{document}

\maketitle
 
\begin{abstract}
Consider an algebraic semigroup $S$ and its closed subscheme 
of idempotents, $E(S)$. When $S$ is commutative, we show that 
$E(S)$ is finite and reduced; if in addition $S$ is irreducible, 
then $E(S)$ is contained in a smallest closed irreducible subsemigroup 
of $S$, and this subsemigroup is an affine toric variety. 
It follows that $E(S)$ (viewed as a partially ordered set) is the set 
of faces of a rational polyhedral convex cone. On the other hand, 
when $S$ is an irreducible algebraic monoid, we show that $E(S)$ 
is smooth, and its connected components are conjugacy classes 
of the unit group.  
\end{abstract}

\section{Introduction}
\label{sec:intro}

This article continues the study of algebraic semigroups and monoids
(not necessarily linear), began in \cite{Rittatore98, Rittatore07}
for monoids and in \cite{Brion12, Brion-Renner} for semigroups.
The idempotents play an essential r\^ole in the structure of 
abstract semigroups; by results of [loc.~cit.], the idempotents
of algebraic semigroups satisfy remarkable existence and finiteness
properties. In this article, we consider the subscheme of idempotents, 
$E(S)$, of an algebraic semigroup $S$ over an algebraically closed 
field; we show that $E(S)$ has a very special structure under
additional assumptions on $S$. Our first main result states:

\begin{theorem}\label{thm:mon}
Let $M$ be an irreducible algebraic monoid, and $G$ its unit group. 
Then the scheme $E(M)$ is smooth, and its connected components
are conjugacy classes of $G$.
\end{theorem}

Note that the scheme of idempotents of an algebraic semigroup is 
not necessarily smooth. Consider indeed an arbitrary variety $X$ 
equipped with the composition law $(x,y) \mapsto x$; then $X$ is 
an algebraic semigroup, and $E(X)$ is the whole $X$. Yet the scheme 
of idempotents is reduced for all examples that we know of; it is 
tempting to conjecture that $E(S)$ is reduced for any algebraic 
semigroup $S$. 

\medskip
When $S$ is commutative, $E(S)$ turns out to be a combinatorial
object, as shown by our second main result:

\begin{theorem}\label{thm:com}
Let $S$ be a commutative algebraic semigroup. Then the scheme 
$E(S)$ is finite and reduced. If $S$ is irreducible, then 
$E(S)$ is contained in a smallest closed irreducible subsemigroup 
of $S$; moreover, this subsemigroup is a toric monoid.
\end{theorem} 

By a toric monoid, we mean an irreducible algebraic monoid $M$ with 
unit group being a torus; then $M$ is affine, as follows e.g. from
\cite[Thm.~2]{Rittatore07}. Thus, $M$ may be viewed as an affine 
toric variety (not necessarily normal). Conversely, every such 
variety has a unique structure of algebraic monoid that extends 
the multiplication of its open torus (see e.g.
\cite[Prop.~1]{Rittatore98}). So we may identify the toric monoids 
with the affine toric varieties. Toric monoids have been studied 
by Putcha under the name of connected diagonal monoids 
(see \cite{Putcha81}); they have also been investigated by Neeb 
in \cite{Neeb92}. 

\medskip

In view of Theorem \ref{thm:com} and of the structure of toric 
monoids, the set of idempotents of any irreducible commutative 
algebraic semigroup, equipped with its natural partial order, 
is isomorphic to the poset of faces of a rational polyhedral 
convex cone. 

\medskip

Theorem \ref{thm:com} extends readily to the case where $S$ has 
a dense subsemigroup generated by a single element;
then $S$ is commutative, but not necessarily irreducible. 
Thereby, one associates an affine toric variety with any point 
of an algebraic semigroup; the corresponding combinatorial data 
may be seen as weak analogues of the spectrum of a linear operator 
(see Example \ref{ex:lin} for details). This construction might
deserve further study.

\medskip

This article is organized as follows. In Subsection 
\ref{subsec:exi}, we present simple proofs of some basic results, 
first obtained in \cite{Brion12, Brion-Renner} by more complicated 
arguments; also, we prove the first assertion of Theorem 
\ref{thm:com}. Subsection \ref{subsec:loc} investigates the local
structure of an algebraic semigroup at an idempotent, in analogy 
with the Peirce decomposition,
\[ R = eRe \oplus (1-e)Re \oplus eR(1-e) \oplus (1-e)R(1-e), \]
of a ring $R$ equipped with an idempotent $e$. As an application,
we show that the isolated idempotents of an irreducible algebraic 
semigroup are exactly the central idempotents (Proposition 
\ref{prop:cent}). In Subsection \ref{subsec:smo}, we obtain
a slightly stronger version of Theorem \ref{thm:mon}, by 
combining our local structure analysis with results of Putcha 
on irreducible linear algebraic monoids (see 
\cite[Chap.~6]{Putcha88}). As an application, we generalize
Theorem \ref{thm:mon} to the intervals in $E(S)$, where $S$ is an 
irreducible algebraic semigroup in characteristic zero 
(Corollary \ref{cor:int}). 

\medskip

We return to commutative semigroups in Subsection 
\ref{subsec:pos}, and show that every irreducible commutative
algebraic semigroup has a largest closed toric submonoid 
(Proposition \ref{prop:bart}). The structure of toric monoids 
is recalled in Subsection \ref{subsec:toric}, and Theorem 
\ref{thm:com} is proved in the case of such monoids. 
The general case is deduced in Subsection \ref{subsec:env}, 
which also contains applications to algebraic semigroups having 
a dense cyclic subsemigroup (Corollary \ref{cor:env}). 

\medskip

In the final Subsection \ref{subsec:fin}, we consider those 
irreducible algebraic semigroups $S$ such that $E(S)$ is finite. 
We first show how to reduce their structure to the case where $S$ 
is a monoid and has a zero; then $S$ is linear in view of 
\cite[Cor.~3.3]{Brion-Rittatore}). Then we present another proof 
of a result of Putcha: any irreducible algebraic monoid having
a zero and finitely many idempotents must have a solvable unit group 
(see \cite[Cor.~10]{Putcha82}, and \cite[Prop.~6.24]{Putcha88}
for a generalization). Putcha also showed that any irreducible 
linear algebraic monoid with nilpotent unit group has finitely many 
idempotents, but this does not extend to solvable unit groups 
(see \cite[Thm.~1.12, Ex.~1.15]{Putcha81}). We refer to work of
Huang (see \cite{Huang96a,Huang96b}) for further results on irreducible 
linear algebraic monoids having finitely many idempotents.

\medskip

\noindent
{\bf Notation and conventions}.
Throughout this article, we consider varieties and schemes over a fixed 
algebraically closed field $k$. We use the textbook \cite{Hartshorne}
as a general reference for algebraic geometry. Unless otherwise stated, 
schemes are assumed to be separated and of finite type over $k$; 
a \emph{variety} is a reduced scheme (in particular, varieties are not 
necessarily irreducible). By a \emph{point} of a variety $X$, we mean 
a $k$-rational point; we identify $X$ with its set of points equipped
with the Zariski topology and with the structure sheaf.

An \emph{algebraic semigroup} is a variety $S$ equipped with an 
associative composition law $\mu : S \times S \to S$. For simplicity, 
we denote $\mu(x,y)$ by $xy$ for any $x,y \in S$. A point 
$e \in S$ is \emph{idempotent} if $e^2 = e$. The set of idempotents 
is equipped with a partial order $\leq$ defined by $e \leq f$ if 
$e = ef = fe$. Also, the idempotents are the $k$-rational points 
of a closed subscheme of $S$: the scheme-theoretic preimage of 
the diagonal under the morphism $S \to S \times S$, 
$x \mapsto (x^2,x)$. We denote that subscheme by $E(S)$.

An \emph{algebraic monoid} is an algebraic semigroup $M$
having a neutral element, $1_M$. The \emph{unit group} of $M$
is the subgroup of invertible elements, $G(M)$; this is an
algebraic group, open in $M$ (see \cite[Thm.~1]{Rittatore98} 
in the case where $M$ is irreducible; the general case follows
easily, see \cite[Thm.~1]{Brion12}).

We shall address some rationality questions for algebraic semigroups,
and use \cite[Chap.~11]{Springer} as a general reference for basic 
rationality results on varieties. As in [loc.~cit.], we fix a subfield 
$F$ of $k$, and denote by $F_s$ the separable closure of $F$ in $k$; 
the Galois group of $F_s$ over $F$ is denoted by $\Gamma$. 
We say that an algebraic semigroup $S$ is \emph{defined over $F$},
if the variety $S$ and the morphism $\mu : S \times S \to S$ are 
both defined over $F$.

\section{The idempotents of an algebraic semigroup}
\label{sec:idem}

\subsection{Existence}
\label{subsec:exi}

We first obtain a simple proof of the following basic result
(\cite[Prop.~1]{Brion12}, proved there by reducing to a finite 
field):

\begin{proposition}\label{prop:exist}
Let $S$ be an algebraic semigroup. Then $S$ has an idempotent.
\end{proposition}

\begin{proof}
Arguing by noetherian induction, we may assume that $S$ has no proper 
closed subsemigroup. As a consequence, the set of powers $x^n$,  
where $n \geq 1$, is dense in $S$ for any $x \in S$; in particular, 
$S$ is commutative. Also, $y S$ is dense in $S$ for any $y \in S$; 
since $yS$ is constructible, it contains a nonempty open subset 
of $S$. Thus, there exists $n = n(x,y) \geq 1$ such that $x^n \in y S$.

Choose $x \in S$. For any $n \geq 1$, let 
\[ S_n := \{ y \in S ~\vert~ x^n \in yS \} .\] 
Then each $S_n$ is a constructible subset of $S$, since $S_n$ is 
the image of the closed subset 
$\{ (y,z) \in S \times S ~\vert~ yz = x^n \} $
under the first projection. Moreover, $S = \bigcup_{n \geq 1} S_n$.

To show that the closed subscheme $E(S)$ is nonempty, we may replace 
$k$ with any larger algebraically closed field, and hence assume that 
$k$ is uncountable. Then, by the next lemma, there exists $n \geq 1$ 
such that $S_n$ contains a nonempty open subset of $S$. Since the set 
of powers $x^{mn}$, where $m \geq 1$, is dense in $S$, it follows that
$S_n$ contains some $x^{mn}$. Equivalently, there exists $z \in S$
such that $x^n = x^{mn} z$. Let $y := x^n$, then $y = y^m z$ and hence
$y^{m-1} = y^{2m-2} z$. Thus, $y^{m-1} z$ is idempotent.
\end{proof}

\begin{lemma}\label{lem:const}
Let $X$ be a variety, and $(X_i)_{i \in I}$ a countable family of 
constructible subsets such that $X = \bigcup_{i \in I} X_i$. 
If $k$ is uncountable, then some $X_i$ contains a nonempty open 
subset of $X$. 
\end{lemma}

\begin{proof}
Since each $X_i$ is constructible, it can be written as a finite 
disjoint union of irreducible locally closed subsets. We may thus 
assume that each $X_i$ is locally closed and irreducible; 
then we may replace $X_i$ with its closure, and thus assume that 
all the $X_i$ are closed and irreducible. We may also replace 
$X$ with any nonempty open subset $U$, and $X_i$ with $X_i \cap U$. 
Thus, we may assume in addition that $X$ is irreducible. We then 
have to show that $X_i = X$ for some $i \in I$.

We now argue by induction on the dimension of $X$. If $\dim(X) = 1$,
then each $X_i$ is either a finite subset or the whole $X$. 
But the $X_i$ cannot all be finite: otherwise, $X$, and hence $k$, 
would be countable. This yields the desired statement.

In the general case, assume that each $X_i$ is a proper subset of $X$.
Since the set of irreducible hypersurfaces in $X$ is uncountable, 
there exists such a hypersurface $Y$ which is not contained 
in any $X_i$. In other words, each $Y \cap X_i$ is a proper subset of 
$Y$. Since $Y = \bigcup_{i \in I} Y \cap X_i$, applying the induction 
assumption to $Y$ yields a contradiction. 
\end{proof}

Next, we obtain refinements of 
\cite[Prop.~4 (iii), Prop.~17 (ii)]{Brion12}, thereby proving
the first assertion of Theorem \ref{thm:com}:

\begin{proposition}\label{prop:com}
Let $S$ be a commutative algebraic semigroup. 

\noindent
{\rm (i)} The scheme $E(S)$ is finite and reduced.

\noindent
{\rm (ii)} $S$ has a smallest idempotent, $e_0$.

\noindent
{\rm (iii)} If the algebraic semigroup $S$ is defined over $F$,
then so is $e_0$.
\end{proposition}

\begin{proof}
(i) It suffices to show that the Zariski tangent space 
$T_e(E(S))$ is zero for any idempotent $e$. 
Since $E(S) = \{ x \in S ~\vert~ x^2 = x \}$ and $S$ is commutative,
we obtain
\[ T_e(E(S)) = \{ z \in T_e(S) ~\vert~ 2 f(z) = z \}, \]
where $f$ denotes the tangent map at $e$ of the multiplication
by $e$ in $S$ (see Lemma \ref{lem:tang} below for details on the 
determination of $T_e(E(S))$ when $S$ is not necessarily commutative). 
Moreover, $f$ is an idempotent endomorphism of the vector space
$T_e(S)$, and hence is diagonalizable with eigenvalues $0$ and $1$. 
This yields the desired vanishing of $T_e(E(S))$. 

(ii) By (i), the subscheme $E(S)$ consists of finitely many points 
$e_1, \ldots, e_n$ of $S$. Their product, $e_1 \cdots e_n =: e_0$, 
satisfies $e_0^2 = e_0$ and $e_0 e_i = e_0$ for $i = 1, \ldots, n$. 
Thus, $e_0$ is the smallest idempotent.

(iii) Assume that the variety $S$ and the morphism $\mu$ are defined 
over $F$; then $E(S)$ is a $F$-subscheme of $S$, and hence a smooth 
$F$-subvariety by (i). In view of \cite[Thm.~11.2.7]{Springer}, 
it follows that $E(S)$ (regarded as a finite subset of $S(k)$) 
is contained in $S(F_s)$; also, $E(S)$ is clearly stable by $\Gamma$. 
Thus, $e_0 \in S(F_s)$ is invariant under $\Gamma$, and hence 
$e_0 \in S(F)$.
\end{proof}

We now deduce from Proposition \ref{prop:com} another fundamental
existence result (which also follows from 
\cite[Thm.~1]{Brion-Renner}):

\begin{corollary}\label{cor:exist}
Let $S$ be an algebraic semigroup defined over $F$. If $S$ 
has an $F$-rational point, then it has an $F$-rational idempotent.
\end{corollary}

\begin{proof}
Let $x \in S(F)$ and denote by $\langle x \rangle$ the smallest 
closed subsemigroup of $S$ containing $x$. Then $\langle x \rangle$ 
is the closure of the set of powers $x^n$, where $n \geq 1$. 
Thus, $\langle x \rangle$ is a commutative algebraic semigroup, 
defined over $F$. So $\langle x \rangle$ contains an idempotent 
defined over $F$, by the previous proposition.
\end{proof}

\subsection{Local structure}
\label{subsec:loc}

In this subsection, we fix an algebraic semigroup $S$
and an idempotent $e \in S$. Then $e$ defines two endomorphisms of 
the variety $S$: the left multiplication, $e_{\ell}: x \mapsto e x$, 
and the right multiplication, $e_r : x \mapsto xe$. Clearly, 
these endomorphisms are commuting idempotents, i.e., they satisfy 
$e_{\ell}^2 = e_{\ell}$, $e_r^2 = e_r$, and $e_{\ell} e_r = e_r e_{\ell}$. 
Since $e_{\ell}$ and $e_r$ fix the point $e$, their tangent maps at 
that point are commuting idempotent endomorphisms, $f_{\ell}$ and $f_r$, 
of the Zariski tangent space $T_e(S)$. Thus, we have a decomposition
into joint eigenspaces
\begin{equation}\label{eqn:decS}
T_e(S) = T_e(S)_{0,0} \oplus T_e(S)_{1,0} 
\oplus T_e(S)_{0,1} \oplus T_e(S)_{1,1},
\end{equation}
where we set
\[ T_e(S)_{a,b} := 
\{ z \in T_e(S) ~\vert~ f_{\ell}(z) = a z, f_r(z) = b z \} \]
for $a,b = 0,1$. The Zariski tangent space of $E(S)$
at $e$ has a simple description in terms of these eigenspaces: 

\begin{lemma}\label{lem:tang}
With the above notation, we have
\begin{equation}\label{eqn:decE} 
T_e(E(S)) = T_e(S)_{1,0} \oplus T_e(S)_{0,1}. 
\end{equation}
Moreover, $T_e(E(S))$ is the image of $f_r - f_{\ell}$.
\end{lemma}

\begin{proof}
We claim that 
\[ T_e(E(S)) = \{ z \in T_e(S) ~\vert~ f_{\ell}(z) + f_r(z) = z \}. \]
Indeed, recall that $E(S)$ is the preimage of the diagonal 
under the morphism $S \to S \times S$, $x \mapsto (\sq(x),x)$,
where $\sq :S \to S$, $x \mapsto x^2$ denotes the square map. 
Thus, we have
\[ T_e(E(S)) = \{ z \in T_e(S) ~\vert~ T_e(\sq)(z) = z \}, \]
where $T_e(\sq)$ denotes the tangent map of $\sq$ at $e$.
Also, $\sq$ is the composition of the diagonal morphism,
$\delta : S \to S \times S$, followed by the multiplication, 
$\mu : S \times S \to S$. Thus, we have
\[ T_e(\sq) = T_{(e,e)}(\mu) \circ T_e(\delta) \]
with an obvious notation. Furthermore,  
$T_e(\delta) : T_e(S) \to T_e(S) \times T_e(S)$ 
is the diagonal embedding; also, 
$T_{(e,e)}(\mu) : T_e(S) \times T_e(S) \to T_e(S)$ equals 
$f_{\ell} \times f_r$, since the restriction of $\mu$ 
to $\{ e \} \times S$ (resp. $S \times \{ e \}$) is just 
$e_{\ell}$ (resp. $e_r$). Thus, $T_e(\sq)  = f_{\ell} + f_r$; 
this proves the claim.

Now (\ref{eqn:decE}) follows readily from the claim together with
the decomposition (\ref{eqn:decS}). For the second assertion, 
let $z \in T_e(S)$ and write $z = z_{0,0} + z_{1,0} + z_{0,1} + z_{1,1}$
in that decomposition. Then $(f_r - f_{\ell})(z) = z_{0,1} - z_{1,0}$
and hence $\Ima(f_r - f_{\ell}) = T_e(S)_{1,0} \oplus T_e(S)_{0,1}$. 
\end{proof}

Next, we observe that each joint eigenspace of $f_{\ell}$ and $f_r$ 
in $T_e(S)$ is the Zariski tangent space to a naturally defined 
closed subsemigroup scheme of $S$. Consider indeed 
the closed subscheme
\[ {_e{}S_e} := \{ x \in S ~\vert~ e x = x e = e \}, \]
where the right-hand side is understood as the scheme-theoretic 
fiber at $e$ of the morphism $e_{\ell} \times e_r : S \to S \times S$.
Define similarly
\[ eS_e := \{ x \in S ~\vert~ e x = x, \; x e = e \}, \quad
{_e{}S e} := \{ x \in S ~\vert~ e x = e, \; x e = x \}, \]
and finally
\[ eSe := \{ x \in S ~\vert~ e x = x e = x \}. \]
Then one readily obtains:

\begin{lemma}\label{lem:sub}
With the above notation, $_e{}S_e$, $eS_e$, $_e{}S e$, and $eSe$ 
are closed subsemigroup schemes of $S$ containing $e$. Moreover,
we have 
\[ T_e({_e{}S_e}) = T_e(S)_{0,0}, \quad T_e(eS_e) = T_e(S)_{1,0},
\quad T_e({_e{}Se}) = T_e(S)_{0,1}, \quad T_e(eSe) = T_e(S)_{1,1}. \]
\end{lemma}

\begin{remarks}\label{rem:sub}
(i) Note that $_e{}S_e$ is the largest closed subsemigroup scheme 
of $S$ containing $e$ as its zero. This subscheme is not 
necessarily reduced, as shown e.g. by \cite[Ex.~3]{Brion12}. 
Specifically, consider the affine space $\bA^3$ equipped with 
pointwise multiplication; this is a toric monoid. 
Let $M$ be the hypersurface of $\bA^3$ with equation 
\[ z^n - x y^n =0, \] 
where $n$ is a positive integer. Then $M$ is a closed toric 
submonoid, containing $e:= (1,0,0)$ as an idempotent. Moreover, 
$_e{}M_e$ is the closed subscheme of $\bA^3$ with ideal generated
by $x-1$ and $z^n -y^n$. Thus, $_e{}M_e $ is everywhere nonreduced 
whenever $n$ is a multiple of the characteristic of $k$ 
(assumed to be nonzero).

\smallskip

\noindent
(ii) Also, note that $eS_e$ is the largest closed subsemigroup 
scheme of $S$ containing $e$ and such that the composition law
is the second projection. (Indeed, every such subsemigroup scheme 
$S'$ satisfies $e x = x$ and $x e = e$ for any $T$-valued point $x$ 
of $S'$, where $T$ is an arbitrary scheme; in other words,
$S' \subset eS_e$. Conversely, for any $T$-valued points $x,y$ 
of $eS_e$, we have $x y = x e y = e y = y$). In particular, $eS_e$ 
consists of idempotents, and $eS_e = x S_x$ for any $k$-rational 
point $x$ of $eS_e$.

Likewise, $_e{}Se$ is the largest closed subsemigroup 
scheme of $S$ containing $e$ and such that the composition 
law is the first projection. We shall see in Corollary 
\ref{cor:red} that $_e{}Se$ and $eS_e$ are in fact reduced.

\smallskip

\noindent
(iii) Finally, $eSe$ is the largest closed submonoid scheme 
of $S$ with neutral element $e$. This subscheme is reduced, 
since it is the image of the morphism 
$S \to S$, $x \mapsto e x e$. Likewise,
$Se$ and $eS$ are closed subsemigroups of $S$, and
\[ T_e(Se) = T_e(S)_{0,1} \oplus T_e(S)_{1,1}, \quad
T_e(eS) = T_e(S)_{1,0} \oplus T_e(S)_{1,1}. \]
\end{remarks}

One would like to have a `global' analogue of the decomposition
(\ref{eqn:decS}) along the lines of the local structure results
for algebraic monoids obtained in \cite{Brion08} (which makes
an essential use of the unit group). Specifically, 
one would like to describe some open neighborhood
of $e$ in $S$ by means of the product of the four pieces
$_e{}S_e$, $_e{}Se$, $eS_e$, and $eSe$ (taken in a suitable order)
and of the composition law of $S$. But this already fails when $S$
is commutative: then both $_e{}Se$ and $eS_e$ consist of the 
reduced point $e$, so that we only have two nontrivial pieces,
$S_e$ and $eS$; moreover, the restriction of the composition law 
to $S_e \times eS \to S$ is just the second projection, since 
$x y = x e y = e y = y$ for all $x \in S_e$ and $y \in eS$. 
Yet we shall obtain global analogues of certain partial sums in 
the decomposition (\ref{eqn:decS}). For this, we introduce 
some notation. 

Consider the algebraic monoid $eSe$ and its unit group, $G(eSe)$.
Since $G(eSe)$ is open in $eSe$, the set
\begin{equation}\label{eqn:U} 
U = U(e) := \{ x \in S ~\vert~ e x e \in G(eSe) \} 
\end{equation}
is open in $S$. Clearly, $U$ contains $e$ and is stable under
$e_{\ell}$ and $e_r$; also, note that 
\[ Ue \cap eU = e U e = G(eSe). \]
We now describe the structure of $Ue$:

\begin{lemma}\label{lem:Ue} 
Keep the above notation.

\noindent
{\rm (i)} $U e = \{ x \in S e ~\vert~ e x \in G(eSe) \}$
and $_e{}U e = {_e{}Se}$.

\noindent
{\rm (ii)} $Ue$ is an open subsemigroup of $Se$.

\noindent 
{\rm (iii)} The morphism 
\[ \varphi : {_e{}S e} \times G(eSe) \longrightarrow S, 
\quad (x,g) \longmapsto x g \]
is a locally closed immersion with image $Ue$. Moreover, 
$\varphi$ is an isomorphism of semigroup schemes, where 
the composition law of the left-hand side is given by
$(x,g)(y,h) := (x,gh)$.

\noindent
{\rm (iv)} The tangent map of $\varphi$ at $(e,e)$ 
induces an isomorphism
\[ T_e(S)_{0,1} \oplus T_e(S)_{1,1} \cong T_e(Ue) = T_e(Se). \]
\end{lemma}

\begin{proof}
Both assertions of (i) are readily checked. The first assertion
implies that $Ue$ is open in $Se$. To show that $Ue$ is a
subsemigroup, note that for any points $x,y$ of $Ue$, we have
$e x y  = e x e y$. Hence $e x y \in G(eSe)$ by (i), so that 
$xy \in Ue$ by (i) again. This completes the proof of (ii).

For (iii), consider $T$-valued points $x$ of $_e{}Se$ and $g$ 
of $G(eSe)$, where $T$ is an arbitrary scheme. 
Then $exge = ege = g$ and hence $xg$ is a $T$-valued point 
of $Ue$. Thus, $\varphi$ yields a morphism
$_e{}Se \times G(eSe) \to Ue$. Moreover, we have
\[ \varphi(x,g) \varphi(y,h) = x g y h = x g e y h 
= x g e h = x g h = \varphi((x,g)(y,h)), \]
that is, $\varphi$ is a homomorphism of semigroup schemes.
To show that $\varphi$ is an isomorphism, consider a $T$-valued
point $z$ of $Ue$ and denote by $(ez)^{-1}$ the inverse of $ez$ 
in $G(eSe)$. Then $z = x g$, where $x := z (ez)^{-1}$ 
and $g := ez$; moreover, $x \in (_e{}S e)(T)$ and 
$g \in (eSe)(T)$. Also, if $z = y h$ where $y \in (_e{}S e)(T)$ 
and $h \in (eSe)(T)$, then $h = e h = e y h = ez$ and 
$y = y e = y h h^{-1} = z (ez)^{-1}$. Thus, the morphism 
\[ Ue \longrightarrow {_e{}Se} \times G(eSe), \quad  
z \longmapsto (z(ez)^{-1}, ez) \] 
is the inverse of $\varphi$.

Finally, (iv) follows readily from (iii) in view of Lemma
\ref{lem:sub}. 
\end{proof}

\begin{corollary}\label{cor:red}
{\rm (i)} The scheme $_e{}Se$ is reduced, and is a union of connected 
components of $E(Se)$. 

\noindent
{\rm (ii)} If $S$ is irreducible, then $_e{}Se$ is the unique irreducible
component of $E(Se)$ through $e$.
\end{corollary}

\begin{proof}
Since $Ue$ is reduced, and isomorphic to $_e{}Se \times G(eSe)$
in view of Lemma \ref{lem:Ue}, 
we see that $_e{}Se$ is reduced as well. Moreover, that lemma
also implies that $_e{}Se = E(Se) \cap Ue$ (as schemes).
In particular, $_e{}Se$ is open in $E(Se)$. But $_e{}Se$ is also
closed; this proves (i).

Next, assume that $S$ is irreducible; then so are $Se$ and $Ue$.
By Lemma \ref{lem:Ue} again, $_e{}Se$ is irreducible as well,
which implies (ii).
\end{proof}

Note that a dual version of Lemma \ref{lem:Ue} yields the structure 
of $eU$; also, $eS_e$ satisfies the dual statement of Corollary 
\ref{cor:red}.

We now obtain a description of the isolated points of $E(S)$ 
(viewed as a topological space). To state our result, denote by 
$C = C(e)$ the union of those irreducible components of $S$ 
that contain $e$, or alternatively, the closure of any neighborhood 
of $e$ in $S$; then $C$ is a closed subsemigroup of $S$.

\begin{proposition}\label{prop:cent}
With the above notation, $e$ is isolated in $E(S)$ if and only 
if $e$ centralizes $C$; then $E(S)$ is reduced at $e$. 

In particular, the isolated idempotents of an irreducible 
algebraic semigroup are exactly the central idempotents.
\end{proposition}

\begin{proof}
Assume that $e$ centralizes $C$; then $e_{\ell}$ and $e_r$ induce 
the same endomorphism of the local ring $\cO_{C,e} = \cO_{S,e}$. 
Thus, $f_{\ell} = f_r$. By Lemma \ref{lem:tang}, it follows that 
$T_e(E(S)) = \{ 0 \}$. Hence $e$ is an isolated reduced point 
of $E(S)$.

Conversely, if $e$ is isolated in $E(S)$, then it is also
isolated in $_e{}Se$ and in $eS_e$ (since they both consist 
of idempotents). As  $_e{}Se$ and $eS_e$ are reduced, it
follows that $T_e(_e{}Se) = \{ 0 \} =T_e(eS_e)$, i.e.,
$T_e(S)_{0,1} = T_e(S)_{1,0} = \{ 0 \}$. In view of Lemma 
\ref{lem:tang}, we thus have $T_e(E(S))= \{ 0 \}$, i.e., 
$E(S)$ is reduced at $e$. Moreover, $f_{\ell} = f_r$ by 
Lemma \ref{lem:tang} again. 
In other words, $e_{\ell}$ and $e_r$ induce the same endomorphism 
of $\fm/\fm^2$, where $\fm$ denotes the maximal ideal of $\cO_{S,e}$. 
Hence $e_{\ell}$ and $e_r$ induce the same endomorphism of
$\fm^n/\fm^{n+1}$ for any integer $n \geq 1$, since the natural map
$\Sym^n(\fm/\fm^2) \to \fm^n/\fm^{n+1}$ is surjective and equivariant 
for the natural actions of $e_{\ell}$ and $e_r$. Next, consider the 
endomorphisms of $\cO_{S,e}/\fm^n$ induced by $e_r$ and $e_{\ell}$: 
these are commuting idempotents of this finite-dimensional $k$-vector 
space, which preserve the filtration by the subspaces
$\fm^m/\fm^n$ ($0 \leq m \leq n$) and coincide on the associated 
graded vector space. Thus, $e_{\ell} = e_r$ as endomorphisms of 
$\cO_{S,e}/\fm^n$ for all $n$, and hence as endomorphisms of $\cO_{S,e}$. 
This means that for any $f \in \cO_{S,e}$ there exists a neighborhood 
$V = V_f$ of $e$ in $S$ such that $f(ex) = f(xe)$ for all $x \in V$. 
Since $\cO_{S,e}$ is the localization of a finitely generated 
$k$-algebra, it follows that we may choose $V$ independently of $f$. 
Then $xe = ex$ for all $x \in V$, and hence for all $x \in C$,
since $V$ is dense in $C$.
\end{proof}

We now return to the decomposition (\ref{eqn:decS}), and obtain
a global analogue of the partial sum 
$T_e(S)_{1,0} \oplus T_e(S)_{0,1} \oplus T_e(S)_{1,1}$
in terms of the open subset $U$:

\begin{lemma}\label{lem:UeU}
{\rm (i)} The morphism
\[ \psi : {_e{}S e} \times G(eSe) \times eS_e \longrightarrow S, 
\quad (x,g,y) \longmapsto x g y \]
is a locally closed immersion with image $U e U$.

\noindent
{\rm (ii)} The tangent map of $\psi$ at $(e,e,e)$ induces 
an isomorphism
\[ T_e(S)_{0,1} \oplus T_e(S)_{1,1} \oplus T_e(S)_{1,0} 
\cong T_e(U e U). \]
\end{lemma}

\begin{proof}
(i) Let $z \in UeU$; then $z = z' z''$ with $z' \in Ue$ and 
$z'' \in eU$. In view of Lemma \ref{lem:Ue}, 
it follows that $z = x g y$ with $x \in {_e{}Se}$, $g \in G(eSe)$, 
and $y \in eS_e$. Then $z e = x g e = x g$; likewise,
$e z = e g y = g y$. Thus, we have $g = e z e$, 
$x = z e (e z e)^{-1}$, and $y = (e z e)^{-1} ez$. In particular, 
$z$ satisfies the following conditions: $z \in U$,
and $z = z e (e z e)^{-1} (eze) (e z e)^{-1} ez$. Conversely, 
if $z \in S$ satisfies the above two conditions, then 
$z \in  _e{}Se \, G(eSe) \, eS_e \subset (Ue) (eU) = UeU$. Also, 
these conditions clearly define a locally closed subset of $S$.
This yields the assertions.

(ii) follows from (i) in view of Lemma \ref{lem:sub}.
\end{proof}

Finally, we obtain a parameterization of those idempotents
of $S$ that are contained in $U e U$. To state it, let
\[ V = V(e) := 
\{ (x,y) \in {_e{}Se} \times eS_e ~\vert~ yx \in G(eSe)\}. \]
Then $V$ is an open neighborhood of $(e,e)$ in $_e{}Se \times eS_e$.
For any point $(x,y)$ of $V$, we denote by $(yx)^{-1}$ the inverse
of $yx$ in $G(eSe)$.

\begin{lemma}\label{lem:idemUeU}
With the above notation, the morphism
\[ \gamma : V \longrightarrow S, \quad 
(x,y) \longmapsto x (yx)^{-1} y \]
induces an isomorphism from $V$ to the scheme-theoretic 
intersection $UeU \cap E(S)$. Moreover, the tangent map of 
$\gamma$ at $(e,e)$ induces an isomorphism 
\[ T_{0,1}(S) \oplus T_{1,0}(S) \cong T_e(UeU \cap E(S)). \]
\end{lemma}

\begin{proof}
We argue with $T$-valued points for an arbitrary scheme $T$, as in 
the proof of Lemma \ref{lem:Ue} (iii). 

Let $z \in U e U$. By Lemma \ref{lem:UeU}, we may write $z$ uniquely
as $x g y$, where $x \in {_e{}Se}$, $g \in G(eSe)$, and $y \in eS_e$.
If $z \in E(S)$, then of course $x g y x g y = x g y$. Multiplying
by $e$ on the left and right, this yields $g y x g = g$ and hence
$g y x = e$. Thus, $(x,y) \in V$ and $z = \gamma(x,y)$. Conversely,
if $(x,y) \in V$, then $x (yx)^{-1} \in {_e{}Se}\, G(eSe)$ and
hence $x (yx)^{-1} \in Ue$ by Lemma \ref{lem:Ue}. Using that lemma
again, it follows that $\gamma(x,y) \in UeU$. Also, one readily 
checks that $\gamma(x,y)$ is idempotent. This shows the first 
assertion, which in turn implies the second assertion.
\end{proof}

\begin{remarks}\label{rem:green}
(i) The above subsets $Ue$, $eU$, $UeU$, and $eUe$ are contained
in the corresponding equivalence classes of $e$ under Green's 
relations (see e.g. \cite[Def.~1.1]{Putcha88} for the definition
of these relations). 

Indeed, for any $x \in Ue$, we have obviously 
$S^1 x \subset S^1 e = S e$, where $S^1$ denotes the monoid 
obtained from $S$ by adjoining a neutral element. Also, $S^1 x$ 
contains $(ex)^{-1} e x = e$. Thus, $S^1 x = S^1 e$, that is,
$x \mathscr{L} e$ with the notation of [loc.~cit.]. Likewise,
$x \mathscr{R} e$ for any $x \in eU$, and $x \mathscr{J} e$ 
for any $x \in UeU$. Finally, $e U e = G(eSe)$ equals the 
$\mathscr{H}$-equivalence class of $e$.

 Also, one readily checks that $eS_e$ (resp. $_e{}Se$) is the 
set of idempotents in the equivalence class of $e$ under
$\mathscr{R}$ (resp. $\mathscr{L}$).

\smallskip

\noindent
(ii) Consider the centralizer of $e$ in $S$,
\[ C_S(e) := \{ x \in S ~\vert~ x e = e x \}. \]
This is a closed subsemigroup scheme of $S$ containing both
$eSe$ and $_e{}S_e$. Moreover, the (left or right) multiplication
by $e$ yields a retraction of semigroup schemes 
$C_S(e) \to eSe$, and we have
\[ T_e C_S(e) = T_e(S)_{0,0} \oplus T_e(S)_{1,1}. \]

We may also consider the left centralizer of $e$ in $S$,
\[ C_S^{\ell}(e) := \{ x \in S ~\vert~ e x = e x e \}. \]
This is again a closed subsemigroup scheme of $S$, 
which contains both $Se$ and $_e{}S_e$. Also,
one readily checks that the left multiplication $e_{\ell}$ 
yields a retraction of semigroup schemes 
$C_S^{\ell}(e) \to eSe$, and we have 
\[ T_e C_S^{\ell}(e) = 
T_e(S)_{0,0} \oplus T_e(S)_{0,1} \oplus T_e(S)_{1,1}. \]
Moreover, $C_S^{\ell}(e) \cap U$ is the preimage of 
$G(eSe)$ under $e_{\ell}$. 

The right centralizer of $e$ in $S$, 
\[ C_S^r(e) := \{ x \in S ~\vert~ x e = e x e \}, \]
satisfies similar properties; note that 
$C_S(e) = C_S^{\ell}(C_S^r(e)) = C_S^r(C_S^{\ell}(e))$.
Also, one easily checks that $U$ is stable under 
$C_S^{\ell}(e) \times C_S^r(e)$ acting on $S$ by left 
and right multiplication.

\smallskip

\noindent
(iii) Recall the description of Green's relations 
for an algebraic monoid $M$ with dense unit group $G$ 
(see \cite[Thm.~1]{Putcha84}). For any $x,y \in M$, we have:
\[ x \mathscr{L} y \Leftrightarrow 
\overline{Mx} = \overline{My} \Leftrightarrow Gx = Gy, \quad
x \mathscr{R} y \Leftrightarrow 
\overline{xM} = \overline{yM} \Leftrightarrow xG = yG, \]
\[ x \mathscr{J} y \Leftrightarrow 
\overline{MxM} = \overline{MyM} \Leftrightarrow GxG = GyG. \]
(Indeed, $x \mathscr{L} y$ if and only if $M x = M y$; then
$\overline{Mx} = \overline{My}$. Since $Gx$ is the unique dense
open $G$-orbit in $\overline{Mx}$, it follows that $Gx = Gy$.
Conversely, if $Gx = Gy$ then $M x = M y$. This proves the
first equivalence; the next ones are checked similarly). 

In view of (i), it follows that $U e \subset G e$ 
for any idempotent $e$ of $M$. Likewise, $e U \subset e G$ 
and hence $U e U \subset G e G$;
also, $e U e \subset e G e$. These inclusions are generally 
strict, e.g., when $M$ is the monoid of $n \times n$ matrices 
and $e \neq 0,1$. Thus, $Ue$ is in general strictly
contained in the $\mathscr{L}$-class of $e$, and likewise
for $eU$, $UeU$. 

Also, in view of (ii), $U$ is stable under left multiplication 
by $C_G^{\ell}(e)$ and right multiplication by $C_G^r(e)$. In 
particular, $Ue$ is an open subset of $Me$ containing 
$C_G^{\ell}(e) e$. If $M$ is irreducible, then $C_G^{\ell}(e) e$ 
is open in $Me$ by \cite[Thm.~6.16 (ii)]{Putcha88}. 
As a consequence, $U$ contains the open 
$C_G^{\ell}(e) \times C_G^r(e)$-stable neighborhood $M_0$ of 
$e$ in $M$, whose structure is described in 
\cite[Thm.~2.2.1]{Brion08}. Yet $M_0$ is in general 
strictly contained in $U$.
\end{remarks}

\subsection{Smoothness}
\label{subsec:smo}

In this subsection, we first obtain a slight generalization 
of Theorem \ref{thm:mon}; we then apply the result to intervals 
in idempotents of irreducible algebraic semigroups.

Recall that an algebraic monoid $M$ is \emph{unit dense}
if it is the closure of its unit group; this holds e.g. when $M$
is irreducible. We may now state:

\begin{theorem}\label{thm:ud}
Let $M$ be a unit dense algebraic monoid, $G$ its unit group, 
and $T$ a maximal torus of $G$; denote by $M^o$ (resp. $G^o$) 
the neutral component of $M$ (resp. of $G$), and by $\overline{T}$
the closure of $T$ in $M$. Then the scheme $E(M)$ is smooth, 
and equals $E(M^o)$. Moreover, the connected components of $E(M)$
are conjugacy classes of $G^o$; every such component meets 
$\overline{T}$. 
\end{theorem}

\begin{proof}
Consider an idempotent $e \in M$, and its open neighborhood 
$U$ defined by (\ref{eqn:U}). Then $UeU$ is a locally closed 
subvariety of $M$ by Lemma \ref{lem:UeU}. We claim that 
$UeU$ is smooth at $e$.

To prove the claim, note that the subset $GeG$ of $M$ is 
a smooth, locally closed subvariety, since it is an orbit 
of the algebraic group $G \times G$ acting on $M$ by left 
and right multiplication. Also, 
$G e G \supset U e U \supset (G \cap U) e (G \cap U)$,
where the first inclusion follows from Remark 
\ref{rem:green} (iii). Moreover, $G \cap U$ is an open 
neighborhood of $e$, dense in $U$ (as $G$ is dense in $M$). 
Since the orbit map $G \times G \to GeG$, 
$(x,y) \mapsto x e y$ is flat, it follows that 
$(G\cap U) e (G \cap U)$ is an open neighborhood of $e$ in 
$G e G$, and hence in $U e U$. This yields the claim.

By that claim together with Lemma \ref{lem:UeU}, the schemes
$_e{}Me$ and $eM_e$ are smooth at $e$. In view of Lemma
\ref{lem:idemUeU}, it follows that $e$ is contained in a smooth,
locally closed subvariety $V$ of $E(M)$ such that
\[ \dim_e(V) = \dim_e({_e{}M e}) + \dim_e(e M_e). \]
Using Lemmas \ref{lem:tang} and \ref{lem:sub}, we obtain
\[ \dim_e(V) = \dim T_e(M)_{0,1} + \dim T_e(M)_{1,0} 
= \dim T_e(E(M)). \] 
Thus, $V$ contains an open neighborhood of $e$ in $E(M)$; 
in particular, $E(M)$ is smooth at $e$. We have shown that 
the scheme $E(M)$ is smooth.

Next, recall that $M^o$ is a closed irreducible submonoid 
of $M$ with unit group $G^o$ (see 
\cite[Prop.~2, Prop.~6]{Brion12}). Also, $E(M) = E(M^o)$ 
as sets, in view of [loc.~cit., Rem.~6 (ii)]. Since 
$E(M)$ is smooth, it follows that $E(M) = E(M^o)$ as schemes.

To complete the proof, we may replace $M$ with $M^o$ and 
hence assume that $M$ is irreducible; 
then $G$ is connected. Thus, $G$ has a largest closed 
connected affine normal subgroup, $G_{\aff}$ (see e.g. 
\cite[Thm.~16, p.~439]{Rosenlicht}). 
Denote by $M_{\aff}$ the closure of $G_{\aff}$ in $M$. 
By \cite[Thm.~3]{Brion12}, $M_{\aff}$ is an irreducible 
affine algebraic monoid with unit group $G_{\aff}$, and 
$E(M) = E(M_{\aff})$ as sets. Thus, we may further assume that 
$M$ is affine, or equivalently linear (see 
\cite[Thm.~3.15]{Putcha88}). Then every conjugacy class in $E(M)$
meets $\overline{T}$ by [loc.~cit., Cor.~6.10].
Moreover, $E(\overline{T})$ is finite by [loc.~cit., Thm.~8.4]
(or alternatively by Proposition \ref{prop:com}). Thus, 
it suffices to check that the $G$-conjugacy class of 
every $e \in E(\overline{T})$ is closed in $M$. 

This assertion is shown in \cite[Lem.~1.2.3]{Brion08} 
under the additional assumption that $k$ has characteristic $0$. 
Yet that assumption is unnecessary; we recall the argument 
for the convenience of the reader. Let $B$ be a Borel subgroup 
of $G$ containing $T$; since $G/B$ is complete, 
it suffices to show that the $B$-conjugacy class of $e$ is closed 
in $M$. But that class is also the $U$-conjugacy class of $e$,
where $U$ denotes the unipotent part of $B$; indeed, we have 
$B = UT$, and $T$ centralizes $e$. So the desired closedness
assertion follows from the fact that all orbits of a unipotent 
algebraic group acting on an affine variety are closed 
(see e.g. \cite[Prop.~2.4.14]{Springer}).
\end{proof}

\begin{remarks}\label{rem:smooth}
(i) If $S$ is a \emph{smooth} algebraic semigroup, 
then the scheme $E(S)$ is smooth as well. 
Indeed, for any idempotent $e$ of $S$,
the variety $Se$ is smooth (since it is the image of the smooth 
variety $S$ under the retraction $e_r$). In view of Lemma 
\ref{lem:Ue}, it follows that $_e{}Se$, and likewise $eS_e$, 
are smooth at $e$. This implies in turn that $E(S)$ is smooth
at $e$, by arguing as in the third paragraph of the proof of 
Theorem \ref{thm:ud}.

\smallskip

\noindent
(ii) In particular, the scheme of idempotents of any 
finite-dimensional associative algebra $A$ is smooth. 
This can be proved directly as follows. Firstly, one reduces 
to the case of an unital algebra: consider indeed the algebra 
$B := k \times A$, where the multiplication is given by 
$(t, x)(u, y) := (tu, t y + u x + x y)$. Then $B$ is a 
finite-dimensional associative algebra with unit $(1,0)$.
Moreover, one checks that the scheme $E(B)$ is the disjoint 
union of two copies of $E(A)$: the images of the morphisms 
$x \mapsto (0,x)$ and $x \mapsto (1,-x)$. Secondly, if $A$ is
unital with unit group $G$, and $e \in A$ is idempotent,
then the tangent map at $1$ of the orbit map
\[ G \longrightarrow A, \quad g \longmapsto g e g^{-1} \]
is identified with $f_r - f_{\ell}$ under the natural 
identifications of $T_1(G)$ and $T_e(A)$ with $A$. 
In view of Lemma \ref{lem:tang}, it follows that 
the conjugacy class of $e$ contains a neighborhood of 
$e$ in $E(A)$; this yields the desired smoothness assertion.

\smallskip

\noindent
(iii) One may ask for a simpler proof of Theorem \ref{thm:ud}
based on a tangent map argument as above. But in the setting of
that theorem, there seems to be no relation between the Zariski 
tangent spaces of $M$ at the smooth point $1_M$ and at the 
(generally singular) point~$e$.
\end{remarks}

Still considering a unit dense algebraic monoid $M$ with unit group $G$,
we now describe the isotropy group scheme of any idempotent $e \in M$
for the $G$-action by conjugation, i.e., the centralizer $C_G(e)$ of $e$ 
in $G$. Recall from Remark \ref{rem:green} (ii) that the centralizer of 
$e$ in $M$ is equipped with a retraction of monoid schemes 
$\tau : C_M(e) \to eMe$; thus, $\tau$ restricts to a retraction of 
group schemes that we still denote by $\tau : C_G(e) \to G(eMe)$.
We may now state the following result, which generalizes 
\cite[Lem.~1.2.2 (iii)]{Brion08} with a more direct proof:

\begin{proposition}\label{prop:iso}
With the above notation, we have an exact sequence of group schemes
\[ 1 \longrightarrow {_e{}G_e} \longrightarrow C_G(e) 
\stackrel{\tau}{\longrightarrow} 
G(eMe) \longrightarrow 1. \]
\end{proposition}

\begin{proof}
Clearly, the scheme-theoretic kernel of $\tau : C_G(e) \to G(eMe)$ 
equals $_e{}G_e$. Since $G(eMe)$ is reduced, it remains to show that 
$\tau$ is surjective on $k$-rational points. For this,
consider the left stabilizer $C_M^{\ell}(e)$ equipped with its reduced
subscheme structure. This is a closed submonoid of $M$; moreover, 
the map
\[ \tau^{\ell} : C_M^{\ell}(e) \longrightarrow  eMe, \quad
x \longmapsto ex \]
is a retraction and a homomorphism of algebraic monoids
(see Remark \ref{rem:green} (ii) again). Also, 
$C_G^{\ell}(e) := C_M^{\ell}(e) \cap G$ is a closed subsemigroup
of $G$, and hence a closed subgroup by  \cite[Exc.~3.5.1.2]{Renner05}. 
This yields a homomorphism of algebraic groups that we still denote by 
$\tau^{\ell} : C_G^{\ell}(e) \to G(eMe)$.

We claim that the latter homomorphism is surjective. Indeed, let 
$x \in G(eMe)$. Then $x M = eM$, since $x \in e M$ and 
$e = x x^{-1} \in x M$. As $x G$ is the unique dense $G$-orbit in 
$xM$ for the $G$-action on $M$ by right multiplication, it
follows that $x G = e G$. Hence there exists $g \in G$ such that
$x = e g$; then $e g e = x e = x = eg$, i.e., $x \in C_G^{\ell}(e)$.
This proves the claim.

Next, observe that $\overline{C_G^{\ell}(e)}$ (closure in $M$)
is a unit dense submonoid of $M$. Moreover, with the notation of
Theorem \ref{thm:ud}, we have 
$e \in \overline{T}$ and hence 
$T \subset C_G(e) \subset C_G^{\ell}(e)$; thus,
$e \in \overline{C_G^{\ell}(e)}$. Therefore,
$G(eMe) = e C_G^{\ell}(e)$ is contained in $\overline{C_G^{\ell}(e)}$ 
as well. Thus, to show the desired surjectivity, we may replace $M$ 
with $\overline{C_G^{\ell}(e)}$. Then we apply the claim to the right 
stabilizer $C_G^r(e)$; this yields the statement, since 
$C_G^r(C_G^{\ell}(e)) = C_G(e)$.   
\end{proof}

Finally, we apply Theorem \ref{thm:ud} to the structure of intervals 
in $E(S)$, where $S$ is an algebraic semigroup. 
Given two idempotents $e_0, e_1 \in S$ such that $e_0 \leq e_1$, 
we consider 
\[ [e_0,e_1] := \{ x \in E(S) ~\vert~ e_0 \leq x \leq e_1 \}. \]
This has a natural structure of closed subscheme of $S$, namely, 
the scheme-theoretic intersection 
$E(S) \cap {_{e_0}{}S}_{e_0} \cap e_1 S e_1$
(since $e_0 \leq x$ if and only if $x \in {_{e_0}{}S}_{e_0}$, 
and $x \leq e_1$ if and only if $x \in e_1 S e_1$). Note that
$e_1 S e_1$ is a closed submonoid of $S$ containing $e_0$. 
Moreover, 
\begin{equation}\label{eqn:inter} 
{_{e_0}{}S}_{e_0} \cap e_1 S e_1 
= {_{e_0}{}(e_1 S e_1)}_{e_0} =: M(e_0,e_1) = M 
\end{equation}
is a closed submonoid scheme of $e_1 S e_1$ with zero $e_0$,
and $[e_0,e_1] = E(M)$ as schemes. We may now state the 
following result, which sharpens and builds on a result
of Putcha (see \cite[Thm.~6.7]{Putcha88}):

\begin{corollary}\label{cor:int}
Keep the above notation, and assume that $k$ has characteristic 
$0$ and $S$ is irreducible. Then $M$ is reduced, affine, 
and unit dense. Moreover, the interval $[e_0,e_1]$ is smooth; 
each connected component of $[e_0,e_1]$ is a conjugacy class 
of $G(M)^o$.
\end{corollary}

\begin{proof}
We may replace $S$ with $e_1 S e_1$; thus, we may assume that
$S$ is an irreducible algebraic monoid, and $e_1 = 1_S$. 
Then $M = {_{e_0} S}_{e_0}$ is reduced, as follows from the
local structure of $S$ at $e_0$ (see \cite{Brion08}); more
specifically, an open neighborhood of $e_0$ in $M$ is isomorphic
to a homogeneous fiber bundle with fiber ${_{e_0}M}_{e_0}$ by
[loc.~cit., Thm.~2.2.1, Rem.~3.1.3]. Moreover, $M$ is affine 
and unit dense by [loc.~cit., Lem.~3.1.4]. The final assertion 
follows from these results in view of Theorem \ref{thm:ud}.
\end{proof}
 
Note that the above statement does not extend to positive 
characteristics. Indeed, $M$ can be nonreduced in that case, 
as shown by the example in Remark \ref{rem:sub} (i).

\section{Irreducible commutative algebraic semigroups}
\label{icas}

\subsection{The finite poset of idempotents}
\label{subsec:pos}

Throughout this subsection, we consider an irreducible commutative
algebraic semigroup $S$. We first record the following easy result:

\begin{proposition}\label{prop:max}
$S$ has a largest idempotent, $e_1$. Moreover, there exists a positive 
integer $n$ such that $x^n \in e_1 S$ for all $x \in S$. 
If $S$ is defined over $F$, then so is $e_1$.
\end{proposition}

\begin{proof}
By the finiteness of $E(S)$ (Proposition \ref{prop:com}) together
with \cite[Cor.~1]{Brion-Renner}, there exist a positive integer 
$n$, an idempotent $e_1 \in S$, and a nonempty open subset $U$ 
of $S$ such that $x^n$ is a unit of the closed submonoid $e_1S$ 
for all $x \in U$. Since $S$ is irreducible, it follows that 
$x^n \in e_1 S$ for all $x \in S$. In particular, $e \in e_1 S$ 
for any $e \in E(S)$, i.e., $e = e_1 e$; hence $e_1$ is the 
largest idempotent.

If $S$ is defined over $F$, then $e_1$ is defined over $F_s$ 
(by Proposition \ref{prop:com} again) and is clearly invariant 
under $\Gamma$. Thus, $e_1 \in E(S)(F)$.
\end{proof}

With the above notation, the unit group $G(e_1 S)$ is a connected 
commutative algebraic group, and hence has a largest subtorus,
$T$. The closure $\overline{T}$ of $T$ in $S$ is a closed toric 
submonoid with neutral element $e_1$ and unit group $T$.
We now gather further properties of $\overline{T}$:

\begin{proposition}\label{prop:bart}
With the above notation, we have:

\noindent
{\rm (i)} $E(S) = E(\overline{T})$.

\noindent
{\rm (ii)} $\overline{T}$ contains every subtorus of $S$.

\noindent
{\rm (iii)} If $S$ is defined over $F$, then so is $\overline{T}$.
\end{proposition}

\begin{proof}
(i) Note that $\overline{T} = e_1 \overline{T}$, and 
$E(S) = e_1 E(S) = E(e_1 S)$. Thus, we may replace $S$ 
with $e_1 S$, and hence assume that $S$ is an irreducible 
commutative algebraic monoid. Then the first assertion follows
from Theorem \ref{thm:ud}.

(ii) Let $S'$ be a subtorus of $S$ (i.e., a locally closed 
subsemigroup which is isomorphic to a torus as an algebraic 
semigroup), and denote by $e$ the neutral element of $S'$. Then 
$S' = e S' \subset e S = e_1 e S \subset e_1 S$.
Hence we may again replace $S$ with $e_1 S$, and assume that
$S$ is an irreducible commutative algebraic monoid.
Now consider the map 
\[ \varphi : S \longrightarrow eS, \quad x \longmapsto xe. \]
Then $\varphi$ is a surjective homomorphism of algebraic monoids. 
Thus, $\varphi$ restricts to a homomorphism of unit groups
$G := G(S) \to G(eS)$; the image of that homomorphism is closed,
and also dense since $G$ is dense in $S$. Thus, $\varphi$ 
sends $G$ onto $G(eS)$. Since $S' = G(e\overline{S'})$ is a closed 
connected subgroup of $G(eS)$, there exists a closed connected 
subgroup $G'$ of $G$ such that $\varphi(G') = S'$. 
Let $G'_{\aff}$ denote the largest closed connected affine subgroup
of $G'$. Since $S'$ is affine, we also have 
$\varphi(G'_{\aff}) = S'$, as follows from the decomposition 
$G' = G'_{\aff} G'_{\ant}$, where $G'_{\ant}$ denotes the largest 
closed subgroup of $G$ such that every homomorphism from $G'_{\ant}$ 
to an affine algebraic group is constant
(see e.g. \cite[Cor.~5, pp. 440--441]{Rosenlicht}). But 
in view of the structure of commutative affine algebraic groups
(see e.g. \cite[Thm.~3.1.1]{Springer}), we have 
$G'_{\aff} =  T' \times U'$, where $T'$ (resp. $U'$) denotes 
the largest subtorus (resp. the largest unipotent subgroup) 
of $G'_{\aff}$. Thus, $\varphi(T') = S'$, that is, $S' = e T'$. 
Since $T' \subset T$ and $e \in \overline{T}$, it follows that 
$S' \subset \overline{T}$.
  
(iii) Since $e_1$ is defined over $F$, so is $e_1 S$.
Hence the algebraic group $G(e_1 S)$ is also defined over $F$, 
by \cite[Prop.~11.2.8 (ii)]{Springer}. In view of 
\cite[Exp.~XIV, Thm.~1.1]{SGA3}, it follows that the largest 
subtorus $T$ of $G(e_1S)$ is defined over $F$ as well. Thus, 
so is $\overline{T}$. 
\end{proof}

\subsection{Toric monoids}
\label{subsec:toric}

As mentioned in the introduction, the toric monoids are 
exactly the affine toric varieties (not necessarily normal). 
For later use, we briefly discuss their structure; details 
can be found e.g. in \cite{Neeb92}, 
\cite[\S 2, \S 3]{Putcha81}, and \cite[\S 3.3]{Renner05}.

The isomorphism classes of toric monoids are in 
a bijective correspondence with the pairs $(\Lambda,\cM)$, 
where $\Lambda$ is a free abelian group and $\cM$ is 
a finitely generated submonoid of $\Lambda$ which generates 
that group. This correspondence assigns to a toric monoid 
$M$ with unit group $T$, the lattice $\Lambda$ of characters 
of $T$ and the monoid $\cM$ of weights of $T$ acting on 
the coordinate ring $\cO(M)$ via its action on $M$ by
multiplication. Conversely, one assigns to a pair 
$(\Lambda,\cM)$, the torus $T := \Hom(\Lambda,\bG_m)$ 
(consisting of all group homomorphisms from $\Lambda$ 
to the multiplicative group) and the monoid 
$M := \Hom(\cM,\bA^1)$ (consisting of all monoid
homomorphisms from $\cM$ to the affine line equipped 
with the multiplication). The coordinate ring $\cO(T)$
(resp. $\cO(M)$) is identified with the group ring
$k[\Lambda]$ (resp. the monoid ring $k[\cM]$).

Via this correspondence, the idempotents of $M$ are 
identified with the monoid homomorphisms 
$\varepsilon : \cM \to \{ 1,0 \}$. 
Any such homomorphism is uniquely determined 
by the preimage of $1$, which is the intersection 
of $\cM$ with a unique face of the cone, $C(\cM)$, 
generated by $\cM$ in the vector space 
$\Lambda \otimes_{\bZ} \bR$. Moreover, every $T$-orbit 
in $M$ (for the action by multiplication) contains 
a unique idempotent. This defines bijective correspondences 
between the idempotents of $M$, the faces of the 
rational, polyhedral, convex cone $C(\cM)$,
and the $T$-orbits in $M$. Via these correspondences, 
the partial order relation on $E(M)$ is identified with
the inclusion of faces, resp. of orbit closures; 
moreover, the dimension of a face is the dimension 
of the corresponding orbit. In particular, there is 
a unique closed orbit, corresponding to the minimal 
idempotent and to the smallest face of $C(\cM)$
(which is also the largest linear subspace contained
in $C(\cM)$).

The toric monoid $M$ is defined over $F$ if and only if 
$\Lambda$ is equipped with a continuous action of 
$\Gamma$ that stabilizes $\cM$ (this is shown e.g. in 
\cite[Prop.~3.2.6]{Springer} for tori; the case of 
toric monoids is handled by similar arguments). 
Under that assumption, the above correspondences are 
compatible with the actions of $\Gamma$.

We also record the following observation:

\begin{lemma}\label{lem:toric}
Let $M$ be a toric monoid, and $S$ a closed irreducible 
subsemigroup of $M$. Then $S$ is a toric monoid.
\end{lemma}

\begin{proof}
Since $S$ is irreducible, there is a unique $T$-orbit in $M$
that contains a dense subset of $S$. Thus, there exists 
a unique idempotent $e_S$ of $M$ such that $S \cap e_S T$ 
is dense in $S$. Then $S \cap e_S T$ is a closed irreducible 
subsemigroup of the torus $e_S T$, and hence a subtorus
in view of \cite[Exc.~3.5.1.2]{Renner05}. 
This yields our assertion.
\end{proof}

Next, consider a toric monoid $M$ with unit group $T$ 
and smallest idempotent $e_0$, so that the closed $T$-orbit 
in $M$ is $e_0 T$. Denote by $T_0$ the reduced
neutral component of the isotropy group $T_{e_0}$, 
and by $\overline{T_0}$ the closure of $T_0$ in $M$.
Clearly, $\overline{T_0}$ is a toric monoid with torus $T_0$
and zero $e_0$; this monoid is the reduced scheme of the monoid
scheme $M(e_0,e_1)$ defined by (\ref{eqn:inter}). We now gather 
further properties of $\overline{T_0}$: 

\begin{proposition}\label{prop:zero}
With the above notation, we have:

\noindent
{\rm (i)} $\dim(T_0)$ equals the length of any maximal 
chain of idempotents in $M$.

\noindent
{\rm (ii)} $\overline{T_0}$ is the smallest closed irreducible 
subsemigroup of $M$ containing $E(M)$.

\noindent
{\rm (iii)} $\overline{T_0}$ is the largest closed irreducible
subsemigroup of $M$ having a zero.

\noindent
{\rm (iv)} If $M$ is defined over $F$, then so is $\overline{T_0}$.
\end{proposition}

\begin{proof}
(i) Since $e_0 T$ is closed in $M$, we easily obtain 
that $e_0 M = e_0 T$. Also, $e_0 T$ is isomorphic to 
the homogeneous space $T/T_{e_0}$, where $T_{e_0}$ 
denotes the (scheme-theoretic) stabilizer of $e_0$ in $T$. 
Thus, the morphism $\varphi: M \to e_0 M$, $x \mapsto e_0 x$
makes $M$ a $T$-homogeneous fiber bundle over $T/T_{e_0}$; 
its scheme-theoretic fiber at $e_0$ is the closure of $T_{e_0}$ 
in $M$. Therefore, we have 
\[ \dim(M) - \dim(e_0 M) = \dim(T_{e_0}) = \dim(T_0). \] 
Thus, $\dim(T_0)$ is the codimension of the smallest face 
of the cone associated with $M$. This also equals the length 
of any maximal chain of faces, and hence of any maximal chain 
of idempotents. 

(ii) Let $S$ be a closed irreducible subsemigroup of $M$
containing $E(M)$. Then $S$ contains the neutral element of $M$;
it follows that $S$ is a submonoid of $M$, and $G(S)$ is 
a subtorus of $T$. By (i), we have $\dim(T_0) = \dim(G(S)_0)$. 
But $G(S)_0 \subset T_0$, and hence equality holds. Taking 
closures, we obtain that $S$ contains $\overline{T_0}$.

(iii) Let $S$ be a closed irreducible subsemigroup of $M$ 
having a zero, $e$. Then $e$ is idempotent; thus,
$e e_0 = e_0$. For any $x \in S$, we have $x e = e$ and hence
$x e_0 = e_0$. Also, $S$ is a toric monoid by Lemma 
\ref{lem:toric}. Thus, the closure $\overline{S T_0}$ is a
toric monoid, stable by $T_0$ and contained in $M_{e_0}$
(the fiber of $\varphi$ at $e$). Since $T$ has finitely 
orbits in $M$, it follows that $T_{e_0}$ has finitely many
orbits in $M_{e_0}$; since $T_0$ is a subgroup of finite index
in $T_{e_0}$, we see that $T_0$ has finitely many orbits in 
$M_{e_0}$ as well. As a consequence, $\overline{S T_0}$
is the closure of a $T_0$-orbit, and thus equals
$\overline{e_S T_0}$ for some idempotent $e_S$ of $M$. 
But $e_S \in \overline{T_0}$, and hence 
$\overline{S T_0} \subset \overline{T_0}$. In particular,
$S \subset \overline{T_0}$.

(iv) It suffices to show that $T_0$ is defined over $F$. 
For this, we use the bijective correspondence
between $F$-subgroup schemes of $T$ and $\Gamma$-stable
subgroups of $\Lambda$, that associates with any such
subgroup scheme $T'$, the character group of the quotient
torus $T/H$; then  $T'$ is a torus if and only if
the corresponding subgroup $\Lambda'$ is saturated in
$\Lambda$, i.e., the quotient group $\Lambda/\Lambda'$ is
torsion-free (these results follow e.g. from
\cite[Exp.~VIII, \S 2]{SGA3}). 
Under this correspondence, the isotropy subgroup scheme 
$T_{e_0}$ is sent to the largest subgroup $\Lambda_{e_0}$ 
of $\Lambda$ contained in the monoid $\cM$ 
(since $\cO(T/T_{e_0}) = \cO(e_0 T) = \cO(e_0 M)$ is the 
subalgebra of $\cO(M)$ generated by the invertible elements 
of that algebra). Thus, $T_0$ corresponds to the smallest 
saturated subgroup $\Lambda_0$ of $\Lambda$ that contains 
$\Lambda_{e_0}$. Clearly, the action of $\Gamma$ on 
$\Lambda$ stabilizes $\Lambda_{e_0}$, and hence $\Lambda_0$.
\end{proof}

\begin{remark}
The above proof can be shortened by using general
structure results for unit dense algebraic monoids
(see \cite[\S 3.2]{Brion12}). We have chosen to present 
more self-contained arguments.
\end{remark}

\subsection{The toric envelope}
\label{subsec:env}

In this subsection, we return to an irreducible commutative 
algebraic semigroup $S$; we denote by $e_0$ (resp. $e_1$) 
the smallest (resp. largest) idempotent of $S$, by $T$ 
the largest subtorus of $G(e_1 S)$, and by $T_0$ the reduced
neutral component of the isotropy subgroup scheme $T_{e_0}$. 
In view of Proposition \ref{prop:bart}, $\overline{T}$ contains 
$E(S)$ and is the largest toric submonoid of $S$; 
we denote that submonoid by $TM(S)$ to emphasize its intrinsic 
nature. We now obtain an intrinsic interpretation of 
$\overline{T_0}$, thereby completing the proof of Theorem 
\ref{thm:com}:

\begin{proposition}\label{prop:env}
With the above notation, $\overline{T_0}$ is the smallest 
closed irreducible subsemigroup of $S$ containing $E(S)$,
and also the largest toric subsemigroup of $S$ having a zero. 
Moreover, $\overline{T_0}$ is defined over $F$ if so is $S$. 
\end{proposition}

\begin{proof}
Let $S'$ be a closed irreducible subsemigroup of $S$ 
containing $E(S)$. Then we have
$E(S) \subset TM(S') \subset TM(S)$ 
by Proposition \ref{prop:bart}. Thus, $TM(S')$ contains
$\overline{T_0}$ in view of Proposition \ref{prop:zero}.
Hence $S' \supset \overline{T_0}$.

Next, let $M$ be a toric subsemigroup of $S$ having a zero.
Then $M \subset TM(S)$ by Proposition \ref{prop:bart},
and hence  $M \subset \overline{T_0}$ by Proposition 
\ref{prop:zero} again. The final assertion follows similarly 
from that proposition.
\end{proof}
 
We denote $\overline{T_0}$ by $TE(S)$, and call it the
\emph{toric envelope} of $E(S)$; we may view $TE(S)$ as 
an algebro-geometric analogue of the finite poset $E(S)$.

\begin{corollary}\label{cor:env}
Let $S$ be an algebraic semigroup, $x$ a point of $S$,
and $\langle x \rangle$ the smallest closed subsemigroup 
of $S$ containing $x$.

\noindent
{\rm (i)} $\langle x \rangle$ contains a largest closed toric 
subsemigroup, $TM(x)$. 

\noindent
{\rm (ii)} $E(\langle x \rangle)$ is contained in a smallest 
closed irreducible subsemigroup of $\langle x \rangle$. 
Moreover, this subsemigroup, $TE(x)$, is a toric monoid. 

\noindent
{\rm (iii)} $TM(x) = TM(x^n)$ and $TE(x) = TE(x^n)$ for 
any positive integer $n$.

\noindent
{\rm (iv)} If the algebraic semigroup $S$ and the point $x$ 
are defined over $F$, then so are $TM(x)$ and $TE(x)$. 
\end{corollary}

\begin{proof}
(i) By \cite[Lem.~1]{Brion-Renner}, there exists a positive 
integer $n$ such that $\langle x^n \rangle$ is irreducible. 
Since $\langle x^n \rangle$ is also commutative, it contains
a largest closed toric subsemigroup by Proposition 
\ref{prop:bart}. But every toric subsemigroup $S$ of 
$\langle x \rangle$ is contained in $\langle x^n \rangle$.
Indeed, we have $S^n \subset \langle x^n \rangle$; moreover,
$S =S^n$, since the $n$th power map of $S$ restricts to 
a finite surjective homomorphism on any subtorus.

(ii) We claim that $E(\langle x \rangle) = E(\langle x^n \rangle)$
for any positive integer $n$. Indeed, $E(\langle x^n \rangle)$ 
is obviously contained in $E(\langle x \rangle)$. For the 
opposite inclusion, note that for any $y \in \langle x \rangle$, 
we have $y^n \in \langle x^n \rangle$. Taking $y$ idempotent 
yields the claim.

Now choose $n$ as in (i); then the desired statement follows 
from Proposition \ref{prop:env} in view of the claim.

(iii) is proved similarly; it implies (iv) in view of 
Propositions \ref{prop:bart} and \ref{prop:env} again.
\end{proof}

\begin{remark}\label{rem:env}
When the algebraic semigroup $S$ is irreducible and commutative, 
we clearly have $TM(x) \subset TM(S)$ and $TE(x) \subset TE(S)$ 
for all $x \in S$. Moreover, if the field $k$ is not locally finite 
(that is, $k$ is not the algebraic closure of a finite field), 
then there exists $x \in S$ such that $TM(x) = TM(S)$ and 
$TE(x) = TE(S)$: indeed, the torus $T = G(TM(S))$ has a point $x$ 
which generates a dense subgroup, and then 
$\langle x \rangle = TM(S)$. 

On the other hand, if $k$ is locally finite, then any algebraic 
semigroup $S$ is defined over some finite subfield $\bF_q$. Hence 
$S$ is the union of the finite subsemigroups $S(\bF_{q^n})$, 
where $n \geq 1$; it follows that $TM(x) = TE(x)$ consists of
a unique point, for any $x \in S$.
\end{remark}

\begin{example}\label{ex:lin}
Let $S$ be a linear algebraic semigroup, i.e., $S$ is isomorphic
to a closed subsemigroup of $\End(V)$ for some finite-dimensional 
vector space $V$. Given $x \in S$, we describe the combinatorial
data of the toric monoid $TM(x)$ in terms of the spectrum of 
$x$ (viewed as a linear operator on $V$).

Consider the decomposition of $V$ into generalized eigenspaces 
for $x$,
\[ V = \bigoplus_{\lambda} V_{\lambda}, \]
where $\lambda$ runs over the spectrum. Since $x$ acts nilpotently 
on $V_0$ and $TM(x) = TM(x^n)$ for any positive integer $n$, 
we may assume that $x$ acts trivially on $V_0$. Then we may 
replace $V$ with $\bigoplus_{\lambda \neq 0} V_{\lambda}$, 
and hence assume that $x \in \GL(V)$. In that case, 
$\langle x \rangle \cap \GL(V)$ is a closed subsemigroup of
the algebraic group $\GL(V)$, and hence a closed subgroup; 
in particular, $\id_V \in \langle x \rangle$. So 
$\langle x \rangle$ is a closed submonoid of $\End(V)$, and
$G(\langle x \rangle) = \langle x \rangle \cap \GL(V)$.
Thus, $TM(x)$ is the closure of the largest subtorus, $T$,
of the commutative linear algebraic group 
$G := G(\langle x \rangle)$.

In view of the Jordan decompositions $x = x_s x_u$ and 
$G = G_s \times G_u$, we see that $T$ is the largest subtorus 
of $G(\langle x_s \rangle)$. Thus, we may replace $x$ with 
$x_s$ and assume that 
\[ x = \diag(\lambda_1,\ldots, \lambda_r), \] 
where $\lambda_i \in k^*$ for $i = 1, \ldots,r$; we may
further assume that $\lambda_1,\ldots,\lambda_r$ are pairwise
distinct. Then $G(\langle x \rangle)$ is contained in the diagonal 
torus $\bG_m^r$, and the character group of the quotient 
torus $\bG_m^r/G(\langle x \rangle)$ is the subgroup of 
$\Hom(\bG_m^r,\bG_m) \cong \bZ^r$ consisting of those
tuples $(a_1,\ldots,a_r)$ such that 
$\prod_{i=1}^r \lambda_i^{a_i} = 1$. Via the correspondence
between closed subgroups of $\bG_m^r$ and subgroups of $\bZ^r$,
it follows that the character group of $G(\langle x \rangle)$
is the subgroup of $k^*$ generated by 
$\lambda_1,\ldots,\lambda_r$. As a consequence, 
\emph{the free abelian group $\Lambda$ associated with 
the toric monoid $TM(x)$ is isomorphic to the subgroup of $k^*$ 
generated by the $n$th powers of the nonzero eigenvalues, 
for $n$ sufficiently divisible} (so that this subgroup is indeed
free).

Moreover, since the coordinate functions generate the algebra
$\cO(\langle x \rangle)$ and are eigenvectors of 
$G(\langle x \rangle)$, we see that 
\emph{the monoid $\cM$ associated with $TM(x)$ is isomorphic 
to the submonoid of $(k, \times)$ generated by the $n$th powers 
of the eigenvalues, for $n$ sufficiently divisible} (so that 
this monoid embeds indeed into a free abelian group).

Next, we describe the idempotents of $\langle x \rangle$,
where $x$ is a diagonal matrix as above.
These idempotents are among those of the subalgebra 
of $\End(V)$ consisting of all diagonal matrices; hence 
they are of the form 
\[ e_I := \sum_{i \in I} e_i, \]
where $I \subset \{ 1,\ldots,r \}$, and $e_i$ denotes the
projection to the $i$th coordinate subspace. To determine
when $e_I \in \langle x \rangle$, we view $\cM$ as the 
monoid generated by $t_1,\ldots,t_r$, with relations of 
the form
\[ \prod_{a \in A} t_i^{a_i} = \prod_{b \in B} t_j^{b_j}, \]
where $A$, $B$ are disjoint subsets of $\{ 1,\ldots,r \}$
(one of them being possibly empty), and $a_i$, $b_j$ are 
positive integers; such relations will be called 
\emph{primitive}. Since 
\[ E(\langle x \rangle) = E(TM(x)) = \Hom(\cM, \{ 1,0 \}), \]
it follows that $e_I \in \langle x \rangle$ if and only if
either $I$ contains $A \cup B$, or $I$ meets the complements
of $A$ and of $B$. 

In particular, the largest idempotent of $\langle x \rangle$
is $e_{1,\ldots,r} = \id_V$ (this may of course be seen directly).
The smallest idempotent is $e_I$, where $i \in I$ if and only
if the $i$th coordinate is invertible on $\langle x \rangle$;
this is equivalent to the existence of a primitive relation of 
the form $\prod_{a \in A} t_i^{a_i} = 1$, where $i \in A$.
\end{example}

\subsection{Algebraic semigroups with finitely many  idempotents}
\label{subsec:fin}

Consider an algebraic semigroup $S$ such that $E(S)$ is finite. 
Then $S$ has a smallest idempotent, as shown by the proof of 
Proposition \ref{prop:com} (ii). Also, when $S$ is irreducible, 
it has a largest idempotent by the proof of Proposition 
\ref{prop:max}. We now obtain criteria for an idempotent of 
an algebraic semigroup to be the smallest or the largest one 
(if they exist): 

\begin{proposition}\label{prop:minmax}
Let $S$ be an algebraic semigroup, and $e \in S$ an idempotent.

\smallskip

\noindent
{\rm (i)} $e$ is the smallest idempotent if and only if
$e$ is central and $eS$ is a group.

\smallskip

\noindent
{\rm (ii)} When $S$ is irreducible, $e$ is the largest idempotent 
if and only if $e$ is central and there exists a positive integer
$n$ such that $x^n \in eS$ for all $x \in S$.
\end{proposition}

\begin{proof}
(i) Assume that $e$ is the smallest idempotent. Then both
$eS_e$ and $_eSe$ consist of the unique point $e$. Since
$e$ is a minimal idempotent, it follows that $SeS = eSe$ 
by \cite[Prop.~5 (ii)]{Brion12}. In particular, $xe \in eSe$ 
for any $x \in S$. As a consequence, $xe = exe$; likewise, 
$ex = exe$ and hence $e$ is central. Thus, $eS = eSe$;
the latter is a group in view of [loc.~cit., Prop.~4 (ii)].

To show the converse implication, let $f \in S$ be an 
idempotent. Then $ef$ is an idempotent of the group $eS$, 
and hence $ef = e = fe$.

(ii) If $e$ is the largest idempotent, then $eS_e$ and
$_eSe$ still consist of the unique point $e$. By Lemma
\ref{lem:Ue}, it follows that $eSe$ contains $Ue$;
likewise, $eSe$ contains $eU$. Since $S$ is irreducible, 
$eSe$ contains both $Se$ and $eS$. Arguing as in (i), 
this yields that $e$ is central. Also, there exists 
a positive integer $n$ such that every $x \in S$ 
satisfies $x^n \in e(x) S$ for some idempotent $e(x)$, 
by \cite[Cor.~1]{Brion-Renner}. Since $e(x) \leq e$, 
it follows that $x^n \in eS$.

For the converse implication, let again $f \in S$ 
be an idempotent. Then $f = f^n \in eS$, and hence
$f = f e = e f$; in other words, $f \leq e$.
\end{proof}

We now show that the structure of an irreducible algebraic
monoid having finitely many idempotents reduces somehow to 
that of a closed irreducible submonoid having a zero and 
the same idempotents:

\begin{proposition}\label{prop:red}
The following conditions are equivalent for an irreducible 
algebraic semigroup $S$:

\noindent
{\rm (i)} $E(S)$ is finite.

\noindent
{\rm (ii)} $S$ has a smallest idempotent, $e_0$, and a largest one,
$e_1$. Moreover, $E(M_S)$ is finite, where $M_S$ denotes the reduced
neutral component of the submonoid scheme $e_1 S_{e_0} \subset S$
(with unit $e_1$ and zero $e_0$).

Under either condition, $E(S)$ is reduced, central in $S$, and
equals $E(M_S)$. 
\end{proposition}

\begin{proof}
(i)$\Rightarrow$(ii) follows from the discussion preceding 
Proposition \ref{prop:minmax}. For the converse, note that
$E(M_S) = E(e_1 S_{e_0})$ as sets, by the definitions of $e_0$ 
and $e_1$; also, $E(e_1 S_{e_0}) = E(M_S)$ as sets, in view of 
Theorem \ref{thm:ud}. 

If $E(S)$ is finite, then it is reduced and central in $S$
by Proposition \ref{prop:cent}. Since $E(S) = E(M_S)$ as sets, 
it follows that this also holds as schemes.
\end{proof}

The above reduction motivates the following statement, 
due to Putcha (see \cite[Cor.~10]{Putcha82}); we present 
an alternative proof, based on Renner's construction of 
the ``largest reductive quotient'' of an irreducible 
linear algebraic monoid.

\begin{proposition}\label{prop:fin}
Let $M$ be an irreducible algebraic monoid having a zero. 
If $E(M)$ is finite, then $G(M)$ is solvable.
\end{proposition}

\begin{proof}
By \cite[Thm.~2.5]{Renner85}, there exist an irreducible
algebraic monoid $M'$ equipped with a homomorphism of algebraic 
monoids $\rho  : M \to M'$ that satisfies the following conditions:

\noindent
(i) $\rho$ restricts to a surjective homomorphism
$G := G(M) \to G(M') =: G'$ with kernel the unipotent radical, 
$R_u(G)$.

\noindent
(ii) $\rho$ restricts to an isomorphism 
$\overline{T} \to \overline{T'}$, where $T$ denotes 
a maximal subtorus of $G$, and $T'$ its image under $\rho$.

As a consequence, $G'$ is reductive, that is, $M'$ is 
a reductive monoid. Also, since each conjugacy class of
idempotents in $M$ (resp. $M'$) meets $\overline{T}$
(resp. $\overline{T'}$) and since every idempotent of $M$
is central, we see that $E(M')$ equals $E(\overline{T'})$ and 
is contained in the center of $M'$.

Let $\cC$ be the cone associated with the toric monoid 
$\overline{T'}$. Then $\cC$ is stable under the action of the
Weyl group $W' := W(G',T')$ on the character group of $T'$.
Recall from Subsection \ref{subsec:toric} that 
$E(\overline{T'})$ is in a bijective correspondence with 
the set of faces of $\cC$; this correspondence is compatible 
with the natural actions of $W'$. But $W'$ acts trivially on 
the idempotents of $\overline{T'}$, since they are all 
central. Thus, $W'$ stabilizes every face of $\cC$. 
It follows that $W'$ fixes pointwise every extremal ray;
hence $W'$ fixes pointwise the whole cone $\cC$. Since the 
interior of $\cC$ is nonempty, $W'$ must be trivial, i.e., 
$G'$ is a torus. Hence $G$ is solvable.
\end{proof}

\medskip

\noindent
{\bf Acknowledgements.} Most of this article was written during 
my participation in February 2013 to the Thematic Program on 
Torsors, Nonassociative Algebras and Cohomological Invariants, 
held at the Fields Institute. I warmly thank the organizers of 
the program for their invitation, and the Fields Institute for 
providing excellent working conditions. I also thank Mohan
Putcha and Wenxue Huang for very helpful e-mail exchanges.

\medskip

{\sc Universit\'e de Grenoble I,
D\'epartement de Math\'ematiques,
Institut Fourier, UMR 5582, 
38402 Saint-Martin d'H\`eres Cedex, France}

\emph{E-mail address}: {\tt Michel.Brion@ujf-grenoble.fr}

\end{document}